\documentclass[12pt]{article}
\usepackage{latexsym,amsfonts,amssymb,amsmath,amsthm,ulem}
\usepackage{graphicx}

\begin{document}
\setlength{\baselineskip}{16pt}

\parindent 0.5cm
\evensidemargin 0cm \oddsidemargin 0cm \topmargin 0cm \textheight 22cm \textwidth 16cm \footskip 2cm \headsep
0cm

\newtheorem{theorem}{Theorem}[section]
\newtheorem{lemma}{Lemma}[section]
\newtheorem{proposition}{Proposition}[section]
\newtheorem{definition}{Definition}[section]
\newtheorem{example}{Example}[section]
\newtheorem{corollary}{Corollary}[section]

\newtheorem{remark}{Remark}[section]
\numberwithin{equation}{section}

\def\p{\partial}
\def\I{\textit}
\def\R{\mathbb R}
\def\C{\mathbb C}
\def\u{\underline}
\def\l{\lambda}
\def\a{\alpha}
\def\O{\Omega}
\def\e{\epsilon}
\def\ls{\lambda^*}
\def\D{\displaystyle}
\def\wyx{ \frac{w(y,t)}{w(x,t)}}
\def\imp{\Rightarrow}
\def\tE{\tilde E}
\def\tX{\tilde X}
\def\tH{\tilde H}
\def\tu{\tilde u}
\def\d{\mathcal D}
\def\aa{\mathcal A}
\def\DH{\mathcal D(\tH)}
\def\bE{\bar E}
\def\bH{\bar H}
\def\M{\mathcal M}
\renewcommand{\labelenumi}{(\arabic{enumi})}

\def\disp{\displaystyle}
\def\undertex#1{$\underline{\hbox{#1}}$}
\def\card{\mathop{\hbox{card}}}
\def\sgn{\mathop{\hbox{sgn}}}
\def\exp{\mathop{\hbox{exp}}}
\def\OFP{(\Omega,{\cal F},\PP)}
\newcommand\JM{Mierczy\'nski}
\newcommand\RR{\ensuremath{\mathbb{R}}}
\newcommand\CC{\ensuremath{\mathbb{C}}}
\newcommand\QQ{\ensuremath{\mathbb{Q}}}
\newcommand\ZZ{\ensuremath{\mathbb{Z}}}
\newcommand\NN{\ensuremath{\mathbb{N}}}
\newcommand\PP{\ensuremath{\mathbb{P}}}
\newcommand\abs[1]{\ensuremath{\lvert#1\rvert}}
\newcommand\normf[1]{\ensuremath{\lVert#1\rVert_{f}}}
\newcommand\normfRb[1]{\ensuremath{\lVert#1\rVert_{f,R_b}}}
\newcommand\normfRbone[1]{\ensuremath{\lVert#1\rVert_{f, R_{b_1}}}}
\newcommand\normfRbtwo[1]{\ensuremath{\lVert#1\rVert_{f,R_{b_2}}}}
\newcommand\normtwo[1]{\ensuremath{\lVert#1\rVert_{2}}}
\newcommand\norminfty[1]{\ensuremath{\lVert#1\rVert_{\infty}}}

\newcommand{\ds}{\displaystyle}

\title{Competing Interactions and Traveling Wave Solutions in Lattice Differential Equations}

\author{
Erik S. Van Vleck and Aijun Zhang\thanks{Email addresses. Erik Van Vleck: evanvleck@math.ku.edu, Aijun Zhang:  azhang@math.ku.edu. Zhang was supported in part by the Robert D. Adams Fund and
Van Vleck was supported in part by NSF grant DMS-1115408.} \\
Department of Mathematics\\
University of Kansas\\
Lawrence, KS 66045\\
U.S.A. }

\date{}
\maketitle

\noindent {\bf Abstract.}
The existence of traveling front solutions to bistable lattice differential equations
in the absence of a comparison principle is studied. The results are in the spirit of
those in Bates, Chen, and Chmaj \cite{BCC}, but are applicable to vector equations and
to more general limiting systems. An abstract result on the persistence of traveling
wave solutions is obtained and is then applied to lattice differential equations with
repelling first and/or second neighbor interactions and to some problems with infinite
range interactions.

\bigskip


\bigskip

\newpage

\section{Introduction}
\setcounter{equation}{0}

We study the existence of traveling wave solutions for lattice differential equations (LDEs) by means of a perturbation argument and Fredholm theory for mixed type functional differential
equations. In particular, we prove persistence of traveling waves for a general class of
lattice differential equations with bistable nonlinearity. Consider the following equation,
\vspace{-.1in}\begin{equation}
\begin{cases}
\label{main-eq}
\ds\dot{u}_{j}=d_{1}(u_{j+1}-2u_{j}+u_{j-1})+d_{2}(u_{j+2}-2u_{j}+u_{j-2})-f_{a}(u_{j}),\quad j \in \ZZ\cr
f_{a}(u)=u(u-a)(u-1),\quad a, d_{1}, d_{2} \in\RR.
\end{cases}
\vspace{-.1in}\end{equation}
Our primary interest is in competing interactions between first and second nearest
neighbors when $d_1<0$ and $d_2<0$. We develop a general technique for continuation of
solutions of vector dissipative lattice differential equations and obtain results on
existence of traveling front solutions
for (\ref{main-eq}) when $d_1<0$, $0<-d_2\ll 1$ and when $d_2<0$ and $|d_1|\ll 1$.

Our contribution is to develop techniques based upon the implicit function theorem that
are applicable for vector equations that are similar to that developed by
Bates, Chen, and Chmaj \cite{BCC} for scalar equations.
Whereas in \cite{BCC} the limiting system is the traveling
wave equation associated with the PDE $u_t = u_{xx} - f(u)$, we consider, through the use
of the Fredholm theory for mixed type functional differential equations \cite{MPJ}, limiting
equations that may correspond to lattice differential equations. Among the chief motivations
in this work (and in \cite{BCC}) for the use of implicit function theorem based techniques is
the desire to handle cases in which there does not exist a comparison principle.

Traveling wave solutions to (\ref{main-eq}) have been extensively studied when
$d_1 >0$ and $d_2=0$. In particular, the work of Weinberger based upon the development
of an abstract comparison principle is applicable to both PDEs and LDEs, although primarily
for monostable as opposed to bistable problems. Zinner proved existence of traveling fronts
using topological fixed point results \cite{Z92} and stability \cite{Z91} in the bistable
case. A general stability theory was developed by Chow, Mallet-Paret, and Shen \cite{CMPS}
and Shen employed comparison principle techniques to prove results on existence, uniqueness,
and stability of traveling fronts in which $f\equiv f(u,t)$ may depend periodically on $t$. More
recently Chen, Guo, and Wu developed a framework for existence, uniqueness, and stability
of bistable equations in periodic media \cite{CGW} and Hupkes and Sandstede \cite{HS10}
prove the existence of traveling pulse solutions for discrete in space Fitz-Hugh Nagumo
equations that occur when coupling a relaxation variable to the discrete Nagumo equation
((\ref{main-eq}) with $d_1>0$ and $d_2=0$).

Associated with traveling waves for (\ref{main-eq}) when $d_1>0$ and
$d_2 = 0$ is the mixed type functional differential equations
\[
-c \varphi'(\xi) = d_1(\varphi(\xi-1)-2\varphi(\xi)+\varphi(\xi+1))-f(\varphi(\xi))
\]
which results from the traveling wave ansatz $u_j(t) = \varphi(j-ct)$. Among the
important contributions to the study of these types of equations is the pioneering
work of Rustichini \cite{Rus1,Rus2}, the development by Mallet-Paret of a Fredholm theory for linear
mixed type function differential equations \cite{MPJ} and its use to understand the
global structure of traveling wave solutions \cite{JMP99a}.
Exponential dichotomies for these equations were investigated in \cite{HSS} and \cite{LMP}
and center manifold theory and Lin's method were developed in \cite{HL2} and \cite{HL3},
respectively.

The case in which $d_1 <0$ and $d_2 =0$ was investigated in \cite{VVV} and
\cite{BVV}. In \cite{VVV} a model was developed for the dynamics of twinned microstructures
that arise in
martensitic phase transformation, e.g., in shape memory alloys, which led to
(\ref{main-eq}) in an overdamped limit. Subsequently,
the bistable nonlinearity $f(u) = u - H(u-a)$, $H$ the Heaviside step function,
was employed and transform
techniques were utilized to determine waveforms and wavespeeds. In \cite{BVV}
the cubic nonlinearity was employed and the problem was converted to a periodic media problem
so that the results of \cite{CGW} could be applied. A wealth of traveling wave solutions
of both bistable and monostable type were revealed. Similar techniques may be used to
determine traveling fronts when $d_2<0$ and $d_1=0$ which results in two decoupled systems
of equations.
In \cite{VVV, BVV} one of the
essential ideas (see also \cite{CMPVV}) was to convert to a system
in terms of odd and even lattice sites. This effectively allows us to consider
connecting orbit problems between vector equilibria as opposed to connecting orbit
problems between time independent spatially periodic solutions. Existence and structure
of traveling fronts for higher space dimension versions of (\ref{main-eq}) was recently
investigated in \cite{HVV} using comparison principle and continuation techniques.

This paper is organized as follows. In section \ref{prelim} we present some of the notation
we will employ and background on Fredholm theory from \cite{MPJ} for linear
mixed type functional differential equations.
In addition, we summarize two approaches to
the existence of traveling wave solutions in lattice differential equations. The first
due to Chen, Guo, and Wu \cite{CGW} provides existence, uniqueness and stability results
for traveling front solutions of (\ref{main-eq}) when $d_1$ and $d_2$ are positive. The
second is due to Bates, Chen, and Chmaj \cite{BCC} and provides existence of traveling front
solutions
when $d_1 + 4d_2 >0$. Section \ref{persist} contains our main results and establishes the
persistence of traveling wave solutions for vector equations.
In particular, we consider systems of lattice equations and allow, under certain
non-restrictive conditions, general limiting systems. In section \ref{twaves} we consider
the application of general results in section \ref{persist} to the existence of traveling
fronts to (\ref{main-eq}) for values of $d_1, d_2$ which even after rewriting as a system
(equivalently in a periodic media) do not possess a comparison principle. We end up with conclusions in section \ref{conclusion}.

\section{Preliminaries and Notation}\label{prelim}
\subsection{Fredholm Alternative for Lattice Differential Equations}
If $X,Y$ are Banach spaces with norms $\|\cdot\|_{X}, \|\cdot\|_{Y}$ respectively, then we let $L(X,Y)$ denote the Banach space of bounded linear operators $T:X \to Y$. Denote the kernel and range of $T \in L(X,Y)$ by
$$K(T)=\{x \in X | Tx=0\} \quad and \ R(T)=\{y \in Y | y=Tx  \  for \  some \ x \in X\}.$$  Recall that T is a \textbf{Fredholm operator} if  T satisfies the following:\\
(i) $K(T)$ is finite dimensional in X;\\
(ii) $R(T)$ is closed and also finite dimensional in Y.\\
The Fredholm index of T is defined as
$$ind(T)=dim(K(T))-codim(R(T)).$$

$\|\cdot\|_{L^{2}}$,$\|\cdot\|_{L^{\infty}}$ and $\|\cdot\|_{H^{1}}$ denote the norms of the spaces $L^{2}(\RR,\RR^{N})$, $L^{\infty}(\RR,\RR^{N})$ and $H^{1}(\RR,\RR^{N})$, respectively. Also,
$\langle f,g \rangle:=\int [f_{1}g_{1}+ ...+f_{N}g_{N}]dx$, where $f=(f_{1},...,f_{N}),g=(g_{1},...,g_{N}) \in L^{2}(\RR,\RR^{N})$. We call $f \perp g$  if and only if $\langle f,g \rangle=0$. $K^{\perp}(T)$ denotes the orthogonal complement of the kernel $K(T)$, that is, $K^{\perp}(T)=\{f \in X: \langle f,g \rangle=0$ for all $g \in K(T)\}$.

In \cite{MPJ}, Mallet-Paret investigated  the Fredholm alternative for the following functional differential equations of mixed type, for $1\leq p\leq \infty$,
\vspace{-.1in}\begin{equation}
\label{Fredholm-eq}
u'(x)=\sum_{j=1}^{N_{1}} A_{j}(x)u(x+r_{j})+f(x), u  \in W^{1,p}(\RR,I^{N})
\vspace{-.1in}\end{equation}
where I is some bounded interval, $r_{1}=0$, and $r_{j}\neq r_{k},1\leq j < k \leq N_{1}, N_{1} \geq 2$.

We may write it as \vspace{-.1in}\begin{equation}
\label{Fredholm-eq1}
u'(x)=L u + f(x),
\vspace{-.1in}\end{equation}

and we have the homogeneous equation  \vspace{-.1in}\begin{equation}
\label{Fredholm-eq2}
u'(x)=L u.
\vspace{-.1in}\end{equation}

If $A_{j}(x)$ is a constant matrix, which  is independent of $x$, we denote it by $A_{j,0}$ and then we may write equation (\ref{Fredholm-eq2}) as
\vspace{-.1in}\begin{equation}
\label{Fredholm-eq3}
u'(x)=L_{0} u.
\vspace{-.1in}\end{equation}

Define $\Delta_{L_{0}}(s)=sI-\sum_{j=1}^{N_{1}} A_{j,0}e^{sr_{j}}$.
We say (\ref{Fredholm-eq3}) is \textbf{hyperbolic}  if $\Delta_{L_{0}}(i \theta)\neq 0, \theta \in \RR$.

In the case, $\ds\lim_{x \to \pm\infty} A_{j}(x) =A^{\pm}_{j,0}$ for $1\leq j \leq N_{1}$. Let $L_{0}^{\pm}u:=\sum_{j=1}^{N_{1}} A_{j,0}^{\pm}u(x+r_{j})$. We say equation (\ref{Fredholm-eq2}) is \textbf{asymptotically hyperbolic} if equations (\ref{Fredholm-eq3}) with replacing $L_{0}$ by $L_{0}^{\pm}$ are hyperbolic.

Define $\Lambda_{L}$ by  \vspace{-.1in}\begin{equation}
\label{F-operator}
\Lambda_{L} u=-c u'(x)-L u.
\vspace{-.1in}\end{equation}

We recall Theorem A in Mallet-Paret's paper \cite{MPJ}:
\begin{theorem}
\label{Fredholm-prop}(See \cite{MPJ})
For each p with $1\leq p\leq \infty$, $\Lambda_{L}$ is a Fredholm operator from $W^{1,p}$ to $L^{p}$ provided that equation $-c u'(x)=L u$ is asymptotically hyperbolic.
\end{theorem}

We note here that for linear mixed type functional differential equations the standard
formula for computation of the Fredholm index is generally not valid, but this is remedied
using the spectral flow formula (see \cite{MPJ} Theorem C).

\subsection{Traveling waves for Bistable Dynamics}
In this subsection, we will state the results of the study of the traveling waves of lattice equations for bistable dynamics in  \cite{CGW} and \cite{BCC}. In \cite{CGW}, consider a  general system of spatially discrete reaction diffusion equations for $u(t)=\{u_{n}(t)\}_{n \in \ZZ}$:
 \vspace{-.1in}\begin{equation}
\label{CGW-eq}
\dot{u}_{n}(t)=\sum_{k}a_{n,k}u_{n+k}(t)+f_{n}(u_{n}(t)), n \in \ZZ, t>0, \vspace{-.1in}
\end{equation}
where the coefficients $a_{n,k}$ are real numbers and have the following assumptions:\\
A1. \textbf{Periodic medium}. There exists a positive integer N such that
$a_{n+N,k} = a_{n,k}$ and $f_{n+N}(\cdot) = f_{n}(\cdot) \in C^{2}(\RR)$ for all $n, k \in \ZZ$.\\
A2. \textbf{Existence of ordered, periodic equilibria}. There exist $\vec{\phi}^{\pm}=\{\phi^{\pm}_{n}\}_{n\in \ZZ}$ such that

$\ds\sum_{k}a_{n,k}\phi^{\pm}_{n+k} + f_{n}(\phi^{\pm}) = 0, \phi^{\pm}_{n+N} = \phi^{\pm}_{n}, \phi_{n}^{-} < \phi^{+}_{n}, n\in \ZZ.$
After an appropriate change-of-variables, the equilibria take the form $\vec{\phi}^{-}= \vec{0}$ and $\vec{\phi}^{+}= \vec{1}$.\\
A3. \textbf{Finite-range interaction}. There exists a positive integer $k_{0}$ such that
$a_{n,k} = 0$ for $|k| > k_{0}$ and for all $n\in \ZZ.$\\
A4. \textbf{Nondecoupledness}. For every integer pair $i \neq j$, there exist integers $i_{0}, i_{1}, . . . , i_{m}$ such
that $i_{0} = i$ and $i_{m} =j$ with $\ds\prod_{s=0}^{m-1}a_{i_{s},i_{s+1}-i_{s}} > 0$.\\
A5. \textbf{Ellipticity}.
$a_{n,k} > 0$ for all $k\neq 0$ and $\ds a_{n,0} =-\sum_{k\neq 0}a_{n,k}< 0, n\in \ZZ.$\\

\begin{theorem} {\it
\label{CGW-prop} (See \cite{CGW}). Assume that $\vec{0}$ and $\vec{1}$ are steady-states and any other N-periodic state
$\vec{\phi}=\{\phi_{n}\}_{n \in \ZZ}$ with $\phi_{n} \in (0, 1)$, if it exists, is unstable. Then the problem (\ref{CGW-eq}) admits a solution
$(c, \vec{w})$ satisfying
$\vec{w}(-\infty) =\vec{0} < \vec{w}(\xi) < \vec{1} = \vec{w}(+\infty)$ for all $\xi \in \RR.$}
\end{theorem}

The following theorem is followed by Bates, Chen and Chmaj's results Theorem 1 in \cite{BCC}. Consider the following system,
 \vspace{-.1in}\begin{equation}
\label{BCC-eq}\begin{cases}
\ds c_{\epsilon} u'-\frac{1}{\epsilon^{2}}\sum_{k>0}\alpha_{k}(u(x+\epsilon k)-2u(x)+u(x-\epsilon k))+f(u)=0,\quad \cr u(\pm \infty)=\pm 1,
\end{cases}
\vspace{-.1in}\end{equation}
under the assumptions:(i) $f\in C^{2}(\RR)$ has exactly three zeros, $-1$, $q \in (-1,1)$ and 1, with $f_{u}(\pm 1)>0$;\\
(ii) $\ds\sum_{k>0}\alpha_{k}k^{2}=1, \sum_{k>0}|\alpha_{k}|k^{2}<\infty$, and $\ds\sum_{k>0}\alpha_{k}(1- cos(kz))\geq 0$ for all $z \in [0,2\pi]$.\\

As $\epsilon \to 0$, we have
\vspace{-.1in}\begin{equation}
\label{BCC-ref-eq}\begin{cases}
c u'-u''+f(u)=0,\quad \cr u(\pm \infty)=\pm 1.
\end{cases}
\vspace{-.1in}\end{equation}
It is well-known that (\ref{BCC-ref-eq}) has a unique traveling wave solution denoted by $(c_{0},\phi_{0})$.
\begin{theorem} {\it
\label{BCC-prop}(See \cite{BCC})
Suppose $c_{0}\neq 0$. Then there exists a positive constant $\epsilon^{*}$ such that for every $\epsilon \in (0,\epsilon^{*})$, the problem (\ref{BCC-eq}) admits a solution $(c_{\epsilon},\phi_{\epsilon})$ satisfying
$\ds\lim_{\epsilon \to 0}(c_{\epsilon},\phi_{\epsilon})=(c_{0},\phi_{0})$  in $\RR \times H^{1}(\RR)$.
}\end{theorem}

Next, with the Fredholm theory in \cite{HVV} and ideas in \cite{BCC}, we will study the existence of traveling waves to vector LDEs in an abstract framework.

\section{Persistence of Traveling Waves to Lattice Differential Equations with Perturbations}
\label{persist}
In this section, our goal is to study the persistence of traveling waves of the lattice differential equations,
\vspace{-.1in}\begin{equation}
\label{Refrence-eq}
\Lambda u + F(u)=0, u(+\infty)=\vec{1}, u(-\infty)=\vec{0};
\vspace{-.1in}\end{equation}
where $\Lambda$ is defined as in (\ref{F-operator}), that is, $\Lambda u=-c u'(x)-\ds\sum_{j=1}^{N_{1}} A_{j}(x)u(x+r_{j})$ with $r_{1}=0$, and $r_{j}\neq r_{k},1\leq j < k \leq N_{1}, N_{1} \geq 2$.
The perturbed system of (\ref{Refrence-eq}) is of the form,
\vspace{-.1in}\begin{equation}
\label{Perturbed-eq}
\Lambda u + \epsilon B u+ F(u)=0, u(+\infty)=\vec{1}, u(-\infty)=\vec{0}
\vspace{-.1in}\end{equation}
where $\epsilon >0$ and $Bu:= \ds\sum_{j=1}^{N_{2}} B_{j}(x)u(x+l_{j})$ with $l_{1}=0$, and $l_{j}\neq l_{k},1\leq j < k \leq N_{2}\leq \infty, N_{2} \geq 2$.
We now give the assumptions for the systems of (\ref{Refrence-eq}) and (\ref{Perturbed-eq}). We make the following assumption for the nonlinear term:

\medskip
\noindent{\bf (H1)}{\it$F_{i} \in C^{2}(\RR)$, with $F_{i}(0)=0, F_{i}(1)=0$ for i=1,..., N.}
 \medskip

We remark that even though our application examples in next section focus on bistable nonlinearity, (H1) is a more general assumption. Assume that

\medskip
\noindent{\bf (H2)}{\it There exists a traveling wave solution connecting $\vec{0}$ and $\vec{1}$ for (\ref{Refrence-eq}).}
 \medskip

We let $(c_{0},\phi_{0})$ be a traveling wave with speed $c_{0}>0$ for (\ref{Refrence-eq}). We make the following assumption for the perturbed term:

\medskip
\noindent{\bf (H3)}{\it B is a bounded operator from $H^{1}(\RR,\RR^{N})$ to $L^{2}(\RR,\RR^{N})$ with $B\vec{0}=0$ and  $B\vec{1}=0$.}
 \medskip

For simplicity, let $\Lambda_{\epsilon}= \Lambda + \epsilon B$. We may write (\ref{Perturbed-eq}) in
\vspace{-.1in}\begin{equation}
\label{Perturbed-eq1}
\Lambda_{\epsilon}+ F(u)=0, u(+\infty)=\vec{1}, u(-\infty)=\vec{0}.
\vspace{-.1in}\end{equation}

It is natural to hope that at least for small $\epsilon$, (\ref{Perturbed-eq1}) also has a traveling wave solution.

Let $\gamma(\phi_{0})$ be a N by N matrix with $\gamma_{ii}=F_{i}'(\phi_{0})$ otherwise  $\gamma_{ij}=0$ for $i \neq j$,  and  $L_{0}^{+}\phi:=c_{0}\phi'- \Lambda_{0} \phi+ \gamma(\phi_{0})\phi $ and $L_{0}^{-}\phi:=-c_{0}\phi'- \Lambda_{0}^{*}\phi+ \gamma(\phi_{0})\phi $, where $\Lambda_{0}^{*}\Psi$ is the adjoint operator of $\Lambda_{0}$.  We assume that

\medskip
\noindent{\bf (H4)} {\it
$L_{0}^{\pm}$ are Fredholm Operators from $H^{1}(\RR,\RR^{N})$ to $L^{2}(\RR,\RR^{N})$.}
\medskip

By Theorem \ref{Fredholm-prop}, if $L_{0}^{\pm}$ are asymptotically hyperbolic then (H4) is satisfied. This is equivalent to check the assumption that $L_{0}^{\pm}$ are hyperbolic at $\pm\infty$.  Note that $\phi_{0}(\infty)=\vec{1}$ and $\phi_{0}(-\infty)=\vec{0}$. Let $\hat{L}_{\infty}^{+}\phi:=c_{0}\phi'- \Lambda_{0} \phi+ \gamma(1)\phi $, $\hat{L}_{-\infty}^{+}\phi:=c_{0}\phi'- \Lambda_{0} \phi+ \gamma(0)\phi $, $\hat{L}_{\infty}^{-}\phi:=-c_{0}\phi'- \Lambda_{0}^{*}\phi+ \gamma(1)\phi $ and $\hat{L}_{-\infty}^{-}\phi:=-c_{0}\phi'- \Lambda_{0}^{*}\phi+ \gamma(0)\phi $. Then (H4) is equivalent to the following:

\medskip
\noindent{\bf ($\hat{H}4$)} {\it
$\hat{L}_{\infty}^{\pm}$ and $\hat{L}_{-\infty}^{\pm}$ are hyperbolic.}
\medskip

In applications, we may use ($\hat{H}4$) instead of (H4) if needed since ($\hat{H}4$) can be more easily verified.

\begin{theorem} {\it
 \label{traveling-wave-thm}
Suppose $c_{0}\neq 0$. Assume $H1-H4$. Then there exists a positive constant $\epsilon^{*}$ such that for every $\epsilon \in (0,\epsilon^{*}]$, the problem \eqref{Perturbed-eq1} admits a solution $(c_{\epsilon},\phi_{\epsilon})$ satisfying
$$\lim_{c_{\epsilon} \to c_{0}}(c_{\epsilon},\phi_{\epsilon})=(c_{0},\phi_{0}).$$
}
\end{theorem}

To prove Theorem \ref{traveling-wave-thm}, with the arguments of perturbation of Fredholm operators, we borrow ideas from \cite{BCC}, which are applicable to vector LDEs. We made assumption (H2) for \eqref{Refrence-eq} instead of giving some specific equation having a traveling wave solution like \eqref{BCC-ref-eq}. Existing literature like Theorem \ref{CGW-prop} in Section 2 that Chen, Guo and Wu proved in \cite{CGW} can provide nice candidates for \eqref{Refrence-eq} satisfying (H2). To verify (H4), the Fredholm alternative theory (See \cite{MPJ} and \cite{HVV}) plays an important role.

Let $X:=H^{1}(\RR,\RR^{N})$. Since $L_{0}^{\pm}$ are Fredholm operators and $dim(K(L_{0}^{\pm}))$ is finite, X can be decomposed by $X=X_{1} \bigoplus X_{2}$ with $X_{2}=K(L_{0}^{+})$.  Let $S:=L_{0}^{+}|_{X_{1}}$ be the restriction of $L_{0}^{+}$ on $X_{1}$. Then we have

\begin{lemma} $L_{0}^{+}$ are surjective from $X$ to $Y_{1}$  with $Y_{1}=R(L_{0}^{+})$ and then $S:X \to Y_{1}$ has a bounded inverse.
\end{lemma}

First we define $X_{\eta}:=\{\phi \in H^{1}(\RR,\RR^{N}):\|\phi\|_{H^{1}}\leq\eta\},$ where $\eta$ will be determined later. Following the ideas as in \cite{BCC}, we let $\phi=\phi_{0}+\psi$ for $\psi \in X_{\eta}$ and formulate the problem as
\vspace{-.1in}\begin{equation}
\label{Reformulation-eq1}
L_{0}^{+}\psi=R(c,\psi),
\vspace{-.1in}\end{equation}
where
\vspace{-.1in}\begin{equation}
L_{0}^{+}\psi= c_{0}\psi'-\Lambda_{0}\psi+\gamma(\phi_0)\psi,
\vspace{-.1in}\end{equation}
\vspace{-.1in}\begin{equation}R(c,\psi)=(c_{0}-c)(\phi_{0}'+\psi')+\epsilon B(\psi+\phi_{0})-N(\phi_{0},\psi),
\vspace{-.1in}\end{equation}
\vspace{-.1in}\begin{equation}N(\phi_{0},\psi)=F(\phi_{0}+\psi)-F(\phi_{0})-\gamma(\phi_0)\psi.
\vspace{-.1in}\end{equation}

In some places, we need study the operator of $L_{\epsilon}^{+}=c_{0}\psi'-\Lambda_{\epsilon}\psi+\gamma(\phi_0)\psi$, and its adjoint $$L_{\epsilon}^{-}\psi=- c_{0}\psi'-\Lambda^{*}_{\epsilon}\psi+\gamma(\phi_0)\psi,$$ where $\Lambda_{\epsilon}^{*}$ is the adjoint operator of $\Lambda_{\epsilon}$.

Since the inverse of $S=L_{0}^{+}|_{X_{1}}$ exists, let $T\psi=(S)^{-1}R(c,\psi)$ and we may rewrite (\ref{Reformulation-eq1}) as
\vspace{-.1in}\begin{equation}
\label{Reformulation-eq2}
T\psi=\psi.
\vspace{-.1in}\end{equation}
Then, we can prove the existence of the traveling wave solutions by showing that there is a fixed point for T.

Let $\psi_{0}^{+}=\phi'_{0}/\|\phi'_{0}\|_{L^{2}}$.
\begin{lemma}
\label{estimate-lm2}
$\,$
\begin{itemize}
\item[(1)] $L_{0}^{+}\psi_{0}^{+}=0$. There exists $\psi_{0}^{-} \in L^{2}$ such that $L_{0}^{-}\psi_{0}^{-}=0$ with $\|\psi_{0}^{-}\|_{L^{2}}=1$. Moreover $\psi_{0}^{\pm} \in H^{1}(\RR,\RR^{N})$.
\item[(2)] 
There exists a positive constant $C_{0}$, which depends only on F, such that
$$\|\phi\|_{H^{1}}\leq C_{0}\|L_{0}^{\pm}\phi\|_{L^{2}}$$ for all $\phi \in H^{1}(\RR,\RR^{N})$ satisfying $\phi \perp \psi_{0}^{\pm}$.
\end{itemize}
\end{lemma}
\begin{proof}
(1) $L_{0}^{+}\psi_{0}^{+}=0$ follows by differentiating the equation and a direct computation. By Theorem \ref{Fredholm-prop}, $dim(K(L_{0}^{+}))=dim(K(L_{0}^{-}))$, and then there exists $\psi_{0}^{-} \in L^{2}$ such that $L_{0}^{-}\psi_{0}^{-}=0$ with $\|\psi_{0}^{-}\|_{L^{2}}=1$. Note that $c_{0}(\psi_{0}^{+})'=\Lambda_{0} \psi_{0}^{+}- \gamma(\phi_0)\psi_{0}^{+} $ and $c_{0}(\psi_{0}^{-})'=- \Lambda_{0}^{*}\psi_{0}^{-}+ \gamma(\phi_0)\psi_{0}^{-}$,
which imply that  $\psi_{0}^{\pm} \in H^{1}(\RR,\RR^{N})$.\\
(2) If not, there would exist sequences $\{\phi_{n}\}_{n=1}^{\infty} \subset H^{1}(\RR,\RR^{N})$ and $\{\psi_{n}\}_{n=1}^{\infty} \subset L^{2}(\RR,\RR^{N})$  such that $L_{0}^{\pm}\phi_{n}=\psi_{n}$, $\psi_{n} \in K^{\perp}(L_{0}^{\pm})$, but $$\|\phi_{n}\|_{H^{1}}\geq n\|L_{0}^{\pm}\phi_{n}\|_{L^{2}}.$$ Without loss of generality, we assume $\|\phi_{n}\|_{H^{1}}=1$. Thus, we have $ \|L_{0}^{\pm}\phi_{n}\|_{L^{2}} \to 0$  in $ L^{2}$  as $n \to \infty$. As $\{\phi_{n}\}_{n=1}^{\infty} \subset H^{1}(\RR,\RR^{N})$ is bounded, there exists a subsequence such that $\phi_{n_{j}} \to u$ in $L^{2}$. As $L_{0}^{\pm}$ are closed,  we have $L_{0}^{\pm} u=0$, which implies that $u \in K(L_{0}^{\pm})$.  On the other hand, by the construction, u is in the orthogonal complement of $L_{0}^{\pm}$ denoted by $K^{\perp}(L_{0}^{\pm})$. Note that $K(L_{0}^{\pm})\cap K^{\perp}(L_{0}^{\pm}) =\{0\}$, then $u \equiv 0$, which contradicts with $\|u\|_{H^{1}}=1$.\\
\end{proof}

Let $c(\psi)$ be the unique constant such that $R(c,\psi)\bot \psi_{0}^{-}$. Thus we have
\begin{lemma}
\label{speed-lm}
$R(c,\psi)\bot \psi_{0}^{-}$ if and only if
$$c(\psi)=c_{0}+\frac{\langle \epsilon B\phi_{0},\psi_{0}^{-} \rangle+\langle\epsilon B\psi,\psi_{0}^{-} \rangle-\langle N(\phi_{0},\psi),\psi_{0}^{-} \rangle}
{\langle \phi'_{0},\psi_{0}^{-} \rangle+\langle \psi',\psi_{0}^{-} \rangle}.$$
\end{lemma}
\begin{proof}
It can be verified by direct computation.
\end{proof}

\begin{lemma}
\label{estimate-lm33}
\begin{itemize}
\item[(1)] There exists a $K_{0}>0$ such that $$|c(\psi)-c_{0}|\leq K_{0};$$
\item[(2)] There exists a $K_{1}>0$ such that $$|c(\psi)-c(\hat{\psi})|\leq K_{1}\|\psi-\hat{\psi}\|_{H^{1}};$$
\item[(3)] There exists some $K_{2}< 1 /C_{0}$ such that
$$\|R\psi-R\hat{\psi}\|_{L^{2}}\leq K_{2} \|\psi-\hat{\psi}\|_{H^{1}}.$$
\end{itemize}
\end{lemma}
\begin{proof}(1). Let $\hat{\delta}:=\frac{1}{2}\langle \phi'_{0},\psi_{0}^{-} \rangle (>0)$ and then for $\psi \in X_{\eta}$, $\langle \phi'_{0},\psi_{0}^{-} \rangle+\langle \psi',\psi_{0}^{-} \rangle>\hat{\delta}$ provided $\eta < \hat{\delta}$. Note that
$$|N^{(i)}(\phi_{0},\psi)|\leq M \eta|\psi^{(i)}|$$ and
$$ |N^{(i)}(\phi_{0},\psi)-N^{(i)}(\phi_{0},\hat{\psi})|\leq M \eta|\psi^{(i)}-\hat{\psi}^{(i)}|,i=1,2,$$ where $M=\ds\max_{1 \leq i \leq N}\{\sup_{|s|\leq 1+\hat{\delta}}|F_{i}''(s)|\}$. Thus we have
\begin{align*}|c(\psi)-c_{0}|&=|\frac{\langle \epsilon B\phi_{0},\psi_{0}^{-} \rangle
+\langle \epsilon B\psi,\psi_{0}^{-} \rangle-\langle N(\phi_{0},\psi),\psi_{0}^{-} \rangle}{\langle \phi'_{0},\psi_{0}^{-} \rangle+\langle \psi',\psi_{0}^{-} \rangle}|\\
&\leq \hat{\delta}^{-1}| \langle \epsilon B\phi_{0},\psi_{0}^{-} \rangle+\langle \epsilon B\psi,\psi_{0}^{-} \rangle-\langle N(\phi_{0},\psi),\psi_{0}^{-} \rangle|\\
&\leq \hat{\delta}^{-1}[\| (\epsilon B\phi_{0})\|_{L^{2}}+\epsilon \|B\| \eta+M \eta^{2}]:=K_{0},
\end{align*}\\
(2). $|c(\psi)-c(\hat{\psi})|\leq K_{1}\|\psi-\hat{\psi}\|_{H^{1}},$ where $K_{1}=\hat{\delta}^{-2}[\| (\epsilon B\phi_{0})\|_{L^{2}}+(\hat{\delta}+\eta)(\|\epsilon B\|+M \eta)].$\\
(3). By (1) and (2),
$\|R\psi-R\hat{\psi}\|_{L^{2}}\leq K_{2} \|\psi-\hat{\psi}\|_{H^{1}},
$ where $K_{2}=\hat{\delta}^{-2}[\hat{\delta}+\eta+\|\phi'_{0}\|_{L^{2}}][\| (\epsilon B\phi_{0})\|_{L^{2}}+(\hat{\delta}+\eta)(\|\epsilon B\|+M \eta)].$ Since $\| (\epsilon B\phi_{0})\|_{L^{2}} \to 0$ as $\epsilon \to 0$, and $\eta \in (0,\hat{\delta})$, by choosing $0< \epsilon^{*}\ll 1$ and appropriate $\eta$ we can make $K_{2}$ small enough such that $C_{0}K_{2}<1$ for any $\epsilon \in (0, \epsilon^{*})$.
\end{proof}

Define
$$T \psi=S^{-1}R(c(\psi),\psi).$$

\begin{proof}[\textbf{Proof of Theorem \ref{traveling-wave-thm}}] For each $\psi$, there are unique $v \in X_{1}$ $w \in X_2$ such that $\psi=v+w$. Hence $\psi$ is a solution of $L_0^+ \psi=R(c(\psi),\psi)$ if and only if $v= S^{-1}R(c(v+w),v+w)$  for $w \in X_2$. If suffices to prove that there exists a $\lambda<1$ such that
$\|T\psi-T\hat{\psi}\|_{H^{1}}\leq \lambda \|\psi-\hat{\psi}\|_{H^{1}}$ for all $\psi, \hat{\psi} \in X_{\eta}$ with $\psi=v+w, \hat{\psi}=\hat{v} +w$ for fixed $w \in X_{2}$. Therefore, for each $(\epsilon,w) \in (0,\epsilon^{*})\times X_{2}$, by the above Lemmas \ref{estimate-lm2} and \ref{estimate-lm33} with $\psi^{+}_{0}=w$, $\|T\psi-T\hat{\psi}\|_{H^{1}}\leq C_{0}\|R\psi-R\hat{\psi}\|_{L^{2}}\leq C_{0}K_{2} \|\psi-\hat{\psi}\|_{H^{1}}.$  $\lambda= C_{0}K_{2}<1$. Hence, for each $(\epsilon,w) \in (0,\epsilon^{*})\times X_{2}$, there exists a unique fixed point $v(x;\epsilon,w) \in X_{1}$. Then $\phi(x;\epsilon,w)=\phi_{0}+v(x;\epsilon,w)+w$ is a traveling wave solution to the perturbed equation.

Next we prove the case with $\epsilon =\epsilon^{*}$. We simply put $\phi(x;\epsilon,w)=\phi_{\epsilon}(x)$. Note that $|\phi_{\epsilon}(\xi)|\leq 1$ and $|\phi'_{\epsilon}(\xi)|\leq \|\Lambda\|+\epsilon^{*}\|B\|+ \ds\max_{0 \leq \xi\leq 1} |F(\xi)|$. By Arzel$\grave{a}-$ Ascoli theorem, there exists a subsequence $u_{\epsilon_{k}}$ that converges uniformly on bounded set. Recall that $$c(\psi)=c_{0}+\frac{\langle \epsilon B\phi_{0},\psi_{0}^{-} \rangle+\langle\epsilon B\psi,\psi_{0}^{-} \rangle-\langle N(\phi_{0},\psi),\psi_{0}^{-} \rangle}
{\langle \phi'_{0},\psi_{0}^{-} \rangle+\langle \psi',\psi_{0}^{-} \rangle}.$$ Let $c_{\epsilon^{*}}=\ds\lim_{\epsilon_{k} \to \epsilon^{*}}c_{\epsilon}.$ This completes the proof.
\end{proof}

\begin{remark}
\label{extensionremark}
Replacing $(c_{0},\phi_{0})$ by $(c_{\epsilon^{*}},\phi_{\epsilon^{*}})$ and following the arguments of the proofs in Theorems \ref{traveling-wave-thm}, $\epsilon^{*}$ can be extended further unless (H4) is not satisfied.\\
\end{remark}

Furthermore, if the dimension of the kernel of $L_0^+$ is 1, we can have that the solution set $TW(\epsilon)=\{\phi(x;\epsilon,w): \phi(x;\epsilon,w)=\phi_{0}(x)+v(x;\epsilon,w)+w(x),w \in X_{2}\}$ is one dimensional. Then we can have the equivalent solution set for $TW(\epsilon)=\{\phi(x;\epsilon,\alpha): \phi(x;\epsilon,\alpha)=\phi_{0}(x)+v(x;\epsilon,\alpha)+\alpha \phi_{0}'(x), \alpha \in \RR\}$. We have the following theorem related to the uniqueness of traveling wave solutions.

\begin{theorem} {\it
 \label{traveling-wave-thm-uniqueness}
Assume that $dim(K(L_0^+)=1$. There exist a neighbourhood $U(\epsilon,\alpha)$ of $0 \in \RR^2$ such that for $(\epsilon,\alpha) \in U(\epsilon,\alpha)$, $\phi(x;\epsilon,\alpha)$ is unique up to translation for every fixed $\epsilon$.}
\end{theorem}
\begin{proof}
Note that $TW(\epsilon)=\{\phi(x;\epsilon,\alpha): \phi(x;\epsilon,\alpha)=\phi_{0}(x)+v(x;\epsilon,\alpha)+\alpha \phi_{0}'(x), \alpha \in \RR\}$. On the other hand, we let $TW_{1}(\epsilon)=\{\phi(x+h;\epsilon,0): h \in \RR\}$. We denote the mapping from $(\epsilon,h)$ to $\alpha$ by $\alpha=g(\epsilon,h)$ such that $$\phi(x+h;\epsilon,0)=\phi_{0}(x)+v(x;\epsilon,\alpha)+\alpha \phi_{0}'(x).$$ We claim that g is one to one with respect to h. Otherwise, we assume there exist $h_{2}>h_{1}$ such that $\phi(x+h_{1};\epsilon,0)=\phi(x+h_{2};\epsilon,0)$. Thus $\phi(x;\epsilon,0)$ is periodic with period $h_{2}-h_{1}$, which contradicts with $\phi(\infty;\epsilon,0)=0$ and $\phi(-\infty;\epsilon,0)=1$. Note that $\phi(x;\epsilon,0) \in TW(\epsilon) \cap TW_{1}(\epsilon)$ and $g(0,0)=0$.
Let $G(x;\epsilon,h,\alpha)=\phi_{0}(x)+v(x;\epsilon,\alpha)+\alpha \phi_{0}'(x)-\phi(x+h;\epsilon,0).$
Note that $v(x;\epsilon,\alpha)=S^{-1}R(c(\psi),\psi)$ with $\psi(x)=v(x;\epsilon,\alpha)+\alpha \phi_{0}'(x)$.
Since $\frac{\partial R(c(\psi),\psi)}{\partial \psi}|_{\psi=0}=0$ implies that $\frac{\partial v(\epsilon,\alpha)}{\partial \alpha}|_{(\epsilon,\alpha)=(0,0)}=0$, we have $G_{\alpha}(x;\epsilon,h,\alpha)|_{(\epsilon,h,\alpha)=(0,0,0)}=\phi_{0}'\neq 0$.
By the Implicit Function Theorem, we have that there exists a neighborhood of $(\epsilon,h)=(0,0)$ such that $\alpha=g(\epsilon,h)$ which is continuous. This completes the proof.
\end{proof}

In \cite{MPJ}, Mallet-Paret provided some sufficient conditions for the one dimensional kernel to scalar LDEs and in \cite{HVV}, Hupkes and the first author of current paper generated the results in \cite{MPJ} to vector LDEs,

\vspace{-.1in}\begin{equation}
\label{HVV-eq}
{u}_{t}(x,t)=\gamma u_{xx}(x,t)+\sum_{j=1}^{N} A_{j}(x)[u(x+r_{j})-u(x)]-f(u(x,t),\rho),
\vspace{-.1in}\end{equation}
where $\gamma \geq 0$ and $\rho \in V \subset \RR$. Assume that,\\
\medskip
\noindent{\bf ($HA$)} {\it
A is irreducible (i.e,it is not similar to a block upper-triangular matrix) and nonnegative.}
\medskip

We assume that (Hf1-3) and (HS1-2), which are listed in Section 2 of \cite{HVV}.

\medskip
\noindent{\bf ($h$)} {\it
The conditions (HA),(Hf1-3), and (HS1-2) are all satisfied with the understanding that $V=\{0\}$ and $f(\cdot;0)=f(\cdot)$.}
\medskip

We remark that, for typical bistable nonlinearity, $f(u)=u(u-\rho)(u-1)$ for $0<\rho<1$, (Hf1-3), and (HS1) are satisfied. (HA) is the key assumption which implies the comparison principle and the existence of principal eigenvalue and corresponding positive eigenfunction. For a mixed type equation, (HA) is not satisfied and we may lose the comparison principle.

\begin{proposition}[\cite{HVV}, Proposition 8.2]
\label{HVV-prop}
Consider the Equation \eqref{HVV-eq} with $|c|>0$ and suppose that (h)
is satisfied. Suppose furthermore that for some $\alpha > 0$ the function $P \in BC(\RR,\RR^n)$ has the asymptotics
$$ |P(\xi)|=O(e^{-\alpha|\xi|}), \xi \to -\infty, |P(\xi)-1|=O(e^{-\alpha|\xi|}), \xi \to \infty.$$ Finally, suppose that that there exists a nontrivial solution $p \in W^{s_{r};\infty}(\RR,\RR^n)$ to \eqref{HVV-eq} that has $P(\xi)>0$ for all
$\xi \in \RR$. Then the operator $\Lambda_{c,\gamma}$ is a Fredholm operator with
$$dim Ker(\Lambda_{c,\gamma})=dim Ker(\Lambda^{*}_{c,\gamma})=codim Range(\Lambda_{c,\gamma})=1.$$ In addition, the element $p \in Ker(\Lambda_{c,\gamma})$ satisfies $p(\xi)>0$ for $\xi \in \RR$ and there exists
$p^* \in Ker(\Lambda^*_{c,\gamma})$ satisfies $p^*(\xi)>0$ for $\xi \in \RR$.
\end{proposition}

\section{Applications:Existence of Traveling Waves for Mixed Type LDEs}
\label{twaves}
We will introduce four examples in this section. In the first three subsections, we consider equation \eqref{main-eq}. In the last subsection, we consider the perturbations of equation \eqref{CGW-eq} with infinity range interactions.
Let $d= d_{1}+4d_{2}$. ${u_{j}}$ is called a stationary solution of \eqref{main-eq} if ${u_{j}}$ satisfies $d_{1}(u_{j+1}-2u_{j}+u_{j-1})+d_{2}(u_{j+2}-2u_{j}+u_{j-2})-f_{a}(u_{j})=0$. ${u_{j}}$ is called a N-Periodic stationary solution of \eqref{main-eq} if ${u_{j}}$ is a stationary solution and $u_{j+N}=u_{j}$.

\subsection{Traveling Waves Connecting 0 and 1}
Define $$A_{\epsilon}u:=\frac{1}{\epsilon^{2}}[d_{1}(u(x+\epsilon)-2u(x)+u(x-\epsilon))+
d_{2}(u(x+2\epsilon)-2u(x)+u(x-2\epsilon))].$$

Consider the following equation,
\vspace{-.1in}\begin{equation}
\begin{cases}
\label{main-eq1}
cu'-A_{\epsilon}u+f_{a}(u)=0,\quad j \in \ZZ\cr
f_{a}(u)=u(u-a)(u-1),\quad \cr u(-\infty)=0,u(\infty)=1\quad.
\end{cases}
\vspace{-.1in}\end{equation}

As $\epsilon \to 0$, we have
\vspace{-.1in}\begin{equation}
\label{ref-eq}\begin{cases}
c u'-d u''+f_{a}(u)=0,\quad \cr
f(u)=u(u-a)(u-1), \quad \cr u(-\infty)=0,u(\infty)=1\quad,
\end{cases}
\vspace{-.1in}\end{equation}
For $d>0$, it is well-known that the equation \eqref{ref-eq} has a unique traveling wave $\phi_{0}$ and the speed $c_{0}$.

Now we consider the system (\ref{main-eq1}). By changing variables, we can make f satisfy (i) in Theorem \ref{BCC-prop}. Equation (\ref{main-eq1}) is a particular case with $k=2$ of the equation (\ref{BCC-eq}). For assumption (ii), $\ds\sum_{k>0}\alpha_{k}(1- \cos(kz))\geq 0$ for all $z \in [0,2\pi]$ is equivalent to the following assumption:
\medskip
$$\noindent{\bf (A1)} \quad {d_{1}+4d_{2}>0.}$$  Thus, if we assume (A1), both assumptions in Theorem \ref{BCC-prop} are satisfied, and we have

\begin{theorem}
{\it
 \label{positive-equilibrium-solu-thm}
Suppose $c_{0}\neq 0$. Assume that (A1) holds. Then there exists a positive constant $\epsilon^{*}$ such that for every $\epsilon \in (0,\epsilon^{*})$, the problem (\ref{main-eq1}) admits a solution $(c_{\epsilon},\phi_{\epsilon})$ satisfying
$\ds\lim_{\epsilon \to 0}(c_{\epsilon},\phi_{\epsilon})=(c_{0},\phi_{0})$  in $\RR \times H^{1}(\RR)$.
}
\end{theorem}
\begin{remark}
\label{wave-rk1} If both $d_{1}$ and $d_{2}$ are positive, then (A1) are satisfied automatically. One of them may be negative in the (A1).
\end{remark}

If (A1) is not satisfied, for $d<0$, we will transform our model to a new one which is in the framework of perturbation method developed in the previous section. In section 4.2, we will consider the case with $d < 0$ but $d_{1}$ dominates $d_{2}$ in the sense $|d_{1}|\gg |d_{2}|$. In section 4.3, we will deal with the case with $d < 0$ but $d_{2}$ dominates $d_{1}$ in the sense $|d_{2}|\gg |d_{1}|$.

\subsection{Traveling Waves Connecting Two 2-Periodic States}
As in the work of Brucal - Hallare and Van Vleck \cite{BVV}, we will use a 2-D transformation. First we write the even and odd nodes of the above equation as $x=\{x_{j}\}_{j \in \ZZ^{N}}$ and $y=\{y_{j}\}_{j \in \ZZ^{N}}$, respectively, and obtain
\vspace{-.1in}\begin{equation}
\label{main-eq220}
\begin{cases}
\ds\dot{x}_{k}=d_{1}(y_{k}-2x_{k}+y_{k-1})+d_{2} (x_{k+1}-2x_{k}+x_{k-1})-f_{a}(x_{k}),\quad j \in \ZZ^N\cr
\ds\dot{y}_{k}=d_{1}(x_{k+1}-2y_{k}+x_{k})+d_{2} (y_{k+1}-2y_{k}+y_{k-1})-f_{a}(y_{k})
\end{cases}
\vspace{-.1in}\end{equation}

To compute the equilibria, define $(x_{\pm},y_{\pm})$ by
$$\lim_{j \to -\infty}(x_{j},y_{j})=(x_{-},y_{-}),\lim_{j \to \infty}(x_{j},y_{j})=(x_{+},y_{+})$$
The equilibria satisfy $E:=\{(x,y) \in \RR^{2} | y=x+\frac{f_{a}(x)}{2d_{1}},f_{a}(x)=-f_{a}(y)\}$. Let
$$v_{j}=\frac{x_{j}-x_{-}}{x_{+}-x_{-}},w_{j}=\frac{y_{j}-y_{-}}{y_{+}-y_{-}}.$$
Then substituting into \eqref{main-eq220} we obtain
\vspace{-.1in}\begin{equation}
\label{main-eq22}
\begin{cases}
\ds\dot{v}_{k}=d_{e}(w_{k}-2v_{k}+w_{k-1})+d_{2} (v_{k+1}-2v_{k}+v_{k-1})-f_{e}(v_{k}),\quad j \in \ZZ^N\cr
\ds\dot{w}_{k}=d_{o}(v_{k+1}-2w_{k}+v_{k})+d_{2} (w_{k+1}-2w_{k}+w_{k-1})-f_{o}(w_{k}),
\end{cases}
\vspace{-.1in}\end{equation}
  where $d_{e}=d_{1}\frac{y_{+}-y_{-}}{x_{+}-x_{-}},d_{o}=d_{1}\frac{x_{+}-x_{-}}{y_{+}-y_{-}}$ and
   $$f_{e}=(x_{+}-x_{-})^{2}f_{a_{e}}(v_{j}), f_{o}=(y_{+}-y_{-})^{2}f_{a_{o}}(w_{j}),$$with
   $$a_{e}=-\frac{f''(x_{-})}{x_{+}-x_{-}}-1,a_{o}=-\frac{f''(y_{-})}{y_{+}-y_{-}}-1.$$

By choosing proper x,y such that  $d_{e},d_{0} >0$. If $d_{2}=0$, this is the case studied
in \cite{BVV}. We remark that the case with $d_{2}=0$ can be easily extended to the case with $d_{2}\geq 0$.

Let $\Psi:=(v,w)^{T}$.
Define $$\Delta_{0}\Psi:=\begin{cases}
\label{main-eq2}
\frac{1}{h^{2}}d_{e}(w(x)-2v(x)+w(x-h)),\quad j \in \ZZ^N\cr
\frac{1}{h^{2}}d_{o}(v(x+h)-2w(x)+v(x)),
\end{cases}$$
and $$\Delta_{\epsilon}\Psi:=\begin{cases}
\label{main-eq22n}
\frac{1}{h^{2}} [d_{e}(w(x)-2v(x)+w(x-h))+\epsilon d_{2} (v(x+h)-2v(x)+v(x-h))],\cr
\frac{1}{h^{2}} [d_{o}(v(x+h)-2w(x)+v(x))+\epsilon d_{2} (w(x+h)-2w(x)+w(x-h))].
\end{cases}$$

If $g(x)=(g_{1}(x),g_{2}(x))^{T} \in H^{1}(\RR,\RR^{2})$, we denote $g'=(g'_{1},g'_{2})^{T}$. Now consider
\vspace{-.1in}\begin{equation}
\label{main-eq2n}
\begin{cases}
c \Psi'-\Delta_{\epsilon}\Psi + F(\Psi)=0,\quad
 \quad \cr \Psi(-\infty)=0,\Psi(\infty)=1\quad,
\end{cases}
\vspace{-.1in}\end{equation}
and
\vspace{-.1in}\begin{equation}
\label{ref-eq2}
\begin{cases}
c \Psi'-\Delta_{0}\Psi + F(\Psi)=0,\quad
 \quad \cr \Psi(-\infty)=0,\Psi(\infty)=1\quad,
\end{cases}
\vspace{-.1in}\end{equation}
  where $d_{e}=d\frac{y_{+}-y_{-}}{x_{+}-x_{-}},d_{o}=d\frac{x_{+}-x_{-}}{y_{+}-y_{-}}$ and
   $F(\Psi)=(f_{e}(v), f_{o}(w))^{T}$.

We can pick those equilibria $(x^{\pm},y^{\pm})$ such that after the transformation, any other 2-periodic state
$\vec{\phi}=\{\phi_{n}\}_{n \in \ZZ}$ with $\phi_{n} \in (0, 1)$, if it exists, is unstable. By Theorem \ref{CGW-prop}, there exists a traveling wave solution $(c_{0},\phi_{0})$ for \eqref{ref-eq2}.
To study the system \eqref{ref-eq2}, we will apply the perturbation arguments in Section 3.1.

Let $\phi=\phi_{0}+\psi$. Following \cite{BCC}, we formulate the problem as
$$L_{0} \psi=R(c,\psi),$$
where
$$L_{0} \psi= c_{0}\psi'-\Delta_{0}\psi+\gamma(\phi_0)\psi,$$
$$R(c,\psi)=(c_{0}-c)(\phi_{0}'+\psi')+(\Delta_{\epsilon}-\Delta_{0})(\phi_{0}+\psi)-N(\phi_{0},\psi),$$
$$N(\phi_{0},\psi)=F(\phi_{0}+\psi)-F(\phi_{0})-\gamma(\phi_0)\psi.$$

To investigate the assumption (H4), we let  $L_{\epsilon}\phi:=c_{\epsilon}\phi'- \Delta_{\epsilon}\phi+ \gamma(\phi_{\epsilon}) \phi $ and $L_{\epsilon}^{*}\phi:=-c_{\epsilon}\phi'- \Delta_{\epsilon}^{*}\phi+ \gamma(\phi_{\epsilon}) \phi $, where $c_{\epsilon}\neq 0$, $\phi_{\epsilon}(\infty)=1$ and $\phi_{\epsilon}(-\infty)=0$, where $$\Delta^{*}_{\epsilon}\Phi:=\begin{cases}
\label{main-eq26}
\frac{1}{h^{2}} [d_{o}w(x)-2 d_{e}v(x)+d_{o}w(x-h)+\epsilon d_{2} (v(x+h)-2v(x)+v(x-h))],\cr
\frac{1}{h^{2}} [d_{e}v(x+h)-2d_{o}w(x)+d_{e}v(x)+\epsilon d_{2} (w(x+h)-2w(x)+w(x-h))].
\end{cases}.$$

\begin{lemma}
\label{main-lm1}
For any $\phi \in H^{1}(\RR,\RR^{2})$, $\langle \Delta_{\epsilon}\phi,\phi' \rangle = 0$.
\end{lemma}
\begin{proof}
The proof follows from the direct computation.
\end{proof}

We assume that

\noindent{\bf (B1)} {\it
$L_{\epsilon}(\pm\infty)$ and $L_{\epsilon}^{*}(\pm\infty)$ are hyperbolic.}

We rewrite operators $L_{\epsilon}(\pm\infty)\phi:=c_{\epsilon}\phi'- \Delta_{\epsilon}\phi+ \gamma^{\pm} \phi $ and $L_{\epsilon}^{*}(\pm\infty)\phi:=-c_{\epsilon}\phi'- \Delta_{\epsilon}^{*}\phi+ \gamma^{\pm} \phi $, where $\gamma^{\pm}=\bigl(\begin{smallmatrix} \gamma^{\pm}_{1} & 0\\
        0 &\gamma^{\pm}_{2}\end{smallmatrix}\bigr)$ for i=1,2 with $\gamma^{+}_{1}=f_{e}'(1),\gamma^{+}_{2}=f_{o}'(1)$ and $\gamma^{-}_{1}=f_{e}'(0),\gamma^{-}_{2}=f_{o}'(0)$.

To determine whether $L_{\epsilon}(\pm\infty)$ or $L_{\epsilon}^{*}(\pm\infty)$ are hyperbolic, we need to determine if $\Upsilon_{L_{\epsilon}(\pm\infty)}(i\theta)\neq 0$ or $\Upsilon_{L_{\epsilon}^{*}(\pm\infty)}(i\theta)\neq 0$ for $\theta \in \RR$, where $\Upsilon_{L_{\epsilon}(\pm\infty)}(s):= det(sI-\ds\Sigma_{i=1}^{3}A_{i}^{+} e^{r_{i} s})$ and $\Upsilon_{L_{\epsilon}^{*}(\pm\infty)}(s):= det(sI-\ds\Sigma_{i=1}^{3}A_{i}^{-} e^{r_{i} s})$ with $A_{1}^{+} =\bigl(\begin{smallmatrix} 0 & \frac{d_{e}}{h^{2}}\\
        0 & 0\end{smallmatrix}\bigr), A_{2}^{+} =\bigl(\begin{smallmatrix} -\frac{2d_{e}}{h^{2}}- \gamma_{1}^{\pm}& \frac{d_{e}}{h^{2}}\\
        \frac{d_{0}}{h^{2}}& -\frac{2d_{0}}{h^{2}}- \gamma_{2}^{\pm}\end{smallmatrix}\bigr), A_{3}^{+} =\bigl(\begin{smallmatrix} 0 & 0\\
        \frac{d_{0}}{h^{2}} & 0\end{smallmatrix}\bigr)$, and $A_{1}^{-}=(A_{3}^{+})^{T}$, $A_{2}^{-}=(A_{2}^{+})^{T}$,$A_{3}^{-}=(A_{1}^{+})^{T}$, where $(A_{i}^{+})^{T}$ is the transpose of $A_{i}^{+}$ for $i=1,2,3$.
First we consider the operator $L_{\epsilon}({\pm}\infty)$. Without loss of generality, we set $h=1$ and $s=i\theta$, we have
       \begin{align*}\Upsilon_{L_{\epsilon}({\pm}\infty)}(s)&=|\bigl(\begin{matrix} s-[\epsilon d_{2}(e^{-s}-2-e^{s})-2d_{e}-\gamma_{1}^{\pm}] & -d_{e}(1+e^{-s})\\
        -d_{o}(1+e^{s}) & s-[\epsilon d_{2}(e^{-s}-2-e^{s})-2d_{o}-\gamma_{2}^{\pm}]\end{matrix}\bigr)|\\
        &=|\bigl(\begin{matrix} s-[2\epsilon d_{2}(cos\theta-1)-2d_{e}-\gamma_{1}^{\pm}] & -d_{e}(1+e^{-s})\\
        -d_{o}(1+e^{s}) & s-[2\epsilon d_{2}(cos\theta-1)-2d_{o}-\gamma_{2}^{\pm}]\end{matrix}\bigr)|\\
        &=-\theta^{2}-i[4\epsilon d_{2}(cos\theta-1)-2d_{e}-\gamma_{1}^{\pm}-2d_{o}-\gamma_{2}^{\pm}]\theta\\
        &\quad +[2\epsilon d_{2}(cos\theta-1)-2d_{e}-\gamma_{1}^{\pm}][2\epsilon d_{2}(cos\theta-1)-2d_{o}-\gamma_{2}^{\pm}]-2d_{e}d_{o}(1+cos\theta).\end{align*}
        The imaginary part $Im(\Upsilon_{L_{\epsilon}({\pm}\infty)}(s))=[4\epsilon d_{2}(cos\theta-1)-2d_{e}-\gamma_{1}^{\pm}-2d_{o}-\gamma_{2}^{\pm}]\theta$.

        The real part $Re(\Upsilon_{L_{\epsilon}({\pm}\infty)}(s))=-\theta^{2}+[2\epsilon d_{2}(cos\theta-1)-2d_{e}-\gamma_{1}^{\pm}][2\epsilon d_{2}(cos\theta-1)-2d_{o}-\gamma_{2}^{\pm}]-2d_{e}d_{o}(1+cos\theta)$.
We have the following lemma.
\begin{lemma}
If $\theta \neq 0$, $\Upsilon_{L_{\epsilon}({\pm}\infty)}(i \theta)\neq 0$.
\end{lemma}
\begin{proof} Let $s=i \theta$.
If $Im(\Upsilon_{L_{\epsilon}({\pm}\infty)}(s))=0$ and $\theta \neq 0$, then $4\epsilon d_{2}(cos\theta-1)-2d_{e}-\gamma_{1}^{\pm}-2d_{o}-\gamma_{2}^{\pm}=0$. Thus,  $2\epsilon d_{2}(cos\theta-1)=d_{e}+d_{o}+1/2\gamma_{1}^{\pm}+1/2\gamma_{2}^{\pm}]$. Plugging into $Re(\Upsilon_{L_{\epsilon}({\pm}\infty)}(s))$, we have
        $Re(\Upsilon_{L_{\epsilon}({\pm}\infty)}(s))=-\theta^{2}+[d_{o}-d_{e}-1/2\gamma_{1}^{\pm}+1/2\gamma_{2}^{\pm}][d_{e}-d_{o}+1/2\gamma_{1}^{\pm}-1/2\gamma_{2}^{\pm}]-2d_{e}d_{o}(1+cos\theta).$  Since $d_{o}d_{e}=d_{1}^{2}$, we have $Re(\Upsilon_{L_{\epsilon}({\pm}\infty)}(s))=-\theta^{2}-[d_{o}-d_{e}-1/2\gamma_{1}^{\pm}+1/2\gamma_{2}^{\pm}]^{2}-2d_{1}^{2}(1+cos\theta)<0.$
\end{proof}

Note that $\Upsilon_{L_{\epsilon}({\pm}\infty)}(s)=\Upsilon_{L_{\epsilon}^{*}({\pm}\infty)}(s)$ since they are symmetric with respect to $d_{o},d_{e}$. Thus, we have the following corollary.
\begin{corollary}
The following are equivalent:
\begin{itemize}
\item[(1)]  $L_{\epsilon}({\pm}\infty)$ are hyperbolic;
\item[(2)]  $L_{\epsilon}^{*}({\pm}\infty)$ are hyperbolic;
\item[(3)]  $\Upsilon_{L_{\epsilon}({\pm}\infty)}(0)\neq 0$ or $\Upsilon_{L_{\epsilon}^{*}({\pm}\infty)}(0)\neq 0$.
\end{itemize}
\end{corollary}

We remark that for some typical bistable nonlinearity, we do have $\Upsilon_{L_{\epsilon}({\pm}\infty)}(0)\neq 0$ or $\Upsilon_{L_{\epsilon}^{*}({\pm}\infty)}(0)\neq 0$.
\begin{lemma}
If $\gamma_{1}^{\pm}\gamma_{2}^{\pm}>0$, $\Upsilon_{L_{\epsilon}({\pm}\infty)}(0)\neq 0$ and $\Upsilon_{L_{\epsilon}^{*}({\pm}\infty)}(0)\neq 0$.
\end{lemma}
\begin{proof}
We have
        \begin{align*}\Upsilon_{L_{\epsilon}({\pm}\infty)}(0)&=|\bigl(\begin{matrix} 2d_{e}+\gamma_{1}^{\pm} & -2d_{e}\\
        -2d_{o} & 2d_{o}+\gamma_{2}^{\pm}\end{matrix}\bigr)|\\
        &=2d_{e}\gamma_{2}^{\pm}+2d_{o}\gamma_{1}^{\pm}+\gamma_{1}^{\pm}\gamma_{2}^{\pm}\\
        &>0.
        \end{align*}
\end{proof}

Recall that  in \ref{ref-eq2}, $f_{e}(u)=(x_{+}-x_{-})^{2}u(u-a_{e})(u-1)$ and $f_{o}(u)=(y_{+}-y_{-})^{2}u(u-a_{o})(u-1)$. As long as $a_{e},a_{o} \in (0,1)$, it is easy to verify that $f'_{e}(1)f'_{o}(1)>0$ and $f'_{e}(0)f'_{o}(0)>0$. Thus, we have $\gamma_{1}^{\pm}\gamma_{2}^{\pm}>0$ and then (B1) or (H4) is satisfied.

Let $c(\psi)$ be the unique constant such that $R(c,\psi)\bot \psi_{0}^{-}$ so that
$$c(\psi)=c_{0}+\frac{\langle \triangle_{\epsilon}\phi_{0}-\triangle_{0}\phi_{0},\psi_{0}^{-} \rangle+\langle (\triangle_{\epsilon}-\triangle_{0})\psi,\psi_{0}^{-} \rangle-\langle N(\phi_{0},\psi),\psi_{0}^{-} \rangle}
{\langle \phi'_{0},\psi_{0}^{-} \rangle+\langle \psi',\psi_{0}^{-} \rangle}.$$
Define
$$T \psi=S^{-1}R(c(\psi),\psi).$$
\begin{theorem}{\it
 \label{Existence-tws-thm}
Suppose $c_{0}\neq 0$. The problem (\ref{ref-eq2}) admits a traveling wave solution.}
\end{theorem}

\begin{proof}
It suffices to verify the assumptions (H1-H4) in Theorem \ref{traveling-wave-thm}. Note that $$f_{e}=(x_{+}-x_{-})^{2}f_{a_{e}}(v_{j}), f_{o}=(y_{+}-y_{-})^{2}f_{a_{o}}(w_{j}).$$ As long as we choose those $f_{e}, f_{o}$ with bistable nonlinearities, (H1) is satisfied. By the choices of equilibria, $d_{e}$ and $d_{o}$ are positive. By Theorem \ref{CGW-prop}, (H2) is satisfied. Let $B=\triangle_{1}-\triangle_{0}$. It is obvious that $B\vec{1}=B\vec{0}=0$. Note that
$$
\|B\phi_{\epsilon}\|_{L^{2}}^{2}=(d_{2})^{2}\int_{\RR}[v(x+h)-2v(x)+v(x-h)]^{2}dx+ (d_{2})^{2}\int_{\RR}[w(x+h)-2w(x)+w(x-h)]^{2}dx\\
$$
Since $v(\infty)=1$ and $v(-\infty)=0$, there exists $M>0$ such that $\int_{|x|>M}[v(x+h)-2v(x)+v(x-h)]^{2}dx <1$ and $\int_{|x|>M}[w(x+h)-2w(x)+w(x-h)]^{2}dx <1$. Thus, $\|B\phi_{\epsilon}\|_{L^{2}}^{2} <\infty$ and (H3) is satisfied.
Note that the assumption (H4) is automatically satisfied. By Remark \ref{extensionremark}, $\epsilon$ can be extended to 1. This completes the proof.
\end{proof}

Up to now we studied the cases under $d>0$ and $d<0$ with $|d_{1}|\gg |d_{2}| $. Next we study the case $d<0$ with $|d_{2}|\gg |d_{1}|$.

\subsection{Traveling Waves Connecting Two 4-Periodic States}
If  $|d_{2}|\gg |d_{1}|$, 2-D transformation may not work because (H4) may not be satisfied. Instead of writing the nodes in ordered pairs, we write the nodes of the equation as $w=\{w_{j}\}_{j \in \ZZ^{N}}$, $x=\{x_{j}\}_{j \in \ZZ^{N}}$, $y=\{y_{j}\}_{j \in \ZZ^{N}}$ and $z=\{z_{j}\}_{j \in \ZZ^{N}}$, respectively. We obtain
\vspace{-.1in}\begin{equation}
\label{main-eq320}
\begin{cases}
\ds\dot{w}_{k}=d_{1}(z_{k-1}-2w_{k}+x_{k})+d_{2} (y_{k-1}-2w_{k}+y_{k})-f_{a}(w_{k}),\cr
\ds\dot{x}_{k}=d_{1}(w_{k}-2x_{k}+y_{k})+d_{2} (z_{k-1}-2x_{k}+z_{k})-f_{a}(x_{k}),\cr
\ds\dot{y}_{k}=d_{1}(x_{k}-2y_{k}+z_{k})+d_{2} (w_{k}-2y_{k}+w_{k+1})-f_{a}(y_{k}),\cr
\ds\dot{z}_{k}=d_{1}(y_{k}-2z_{k}+w_{k+1})+d_{2} (x_{k}-2z_{k}+x_{k+1})-f_{a}(z_{k}),\quad k \in \ZZ^N.
\end{cases}
\vspace{-.1in}\end{equation}

To compute the equilibria, define $(w_{\pm},x_{\pm},y_{\pm},z_{\pm})$ by
$$\lim_{j \to -\infty}(w_{j},x_{j},y_{j},z_{j})=(w_{-},x_{-},y_{-},z_{-}),\lim_{j \to \infty}(w_{j},x_{j},y_{j},z_{j})=(w_{+},x_{+},y_{+},z_{+}).$$

The equilibria $(w,x,y,z) \in \RR^{4}$ satisfies
\vspace{-.1in}\begin{equation}
\label{main-eq321}
\begin{cases}
\ds d_{1}(z-2w+x)+2d_{2} (y-w)=f_{a}(w),\cr
\ds d_{1}(w-2x+y)+2d_{2} (z-x)=f_{a}(x),\cr
\ds d_{1}(x-2y+z)+2d_{2} (w-y)=f_{a}(y),\cr
\ds d_{1}(y-2z+w)+2d_{2} (x-z)=f_{a}(z).
\end{cases}
\vspace{-.1in}\end{equation}

Let
$$\hat{w}_{j}=\frac{w_{j}-w_{-}}{w_{+}-w_{-}}, \hat{x}_{j}=\frac{x_{j}-x_{-}}{x_{+}-x_{-}},\hat{y}_{j}=\frac{y_{j}-y_{-}}{y_{+}-y_{-}}, \hat{z}_{j}=\frac{z_{j}-z_{-}}{z_{+}-z_{-}}.$$
Then substituting into \eqref{main-eq320}, for simplicity,  we discard the hats for w,x,y,z and obtain
\vspace{-.1in}\begin{equation}
\label{main-eq322}
\begin{cases}
\ds\dot{w}_{k}=a_{14}z_{k-1}-b_{11}w_{k}+b_{12}x_{k}+a_{13}y_{k-1}+b_{13} y_{k}-f_{1}(w_{k}),\cr
\ds\dot{x}_{k}=b_{21}w_{k}-b_{22}x_{k}+b_{23}y_{k}+a_{24}z_{k-1}+b_{24}z_{k}-f_{2}(x_{k}),\cr
\ds\dot{y}_{k}=b_{32}x_{k}-b_{33}y_{k}+b_{34}z_{k}+b_{31}w_{k}+c_{31}w_{k+1}-f_{3}(y_{k}),\cr
\ds\dot{z}_{k}=b_{43}y_{k}-b_{44}z_{k}+c_{41}w_{k+1}+b_{42}x_{k}+c_{42}x_{k+1}-f_{4}(z_{k}),\quad k \in \ZZ^N,
\end{cases}
\vspace{-.1in}\end{equation}
where $A_{1}=(a_{ij})$, $A_{2}=(b_{ij})$ and $A_{3}=(c_{ij})$ are 4 by 4 matrices given by

$$A_{1}=\left( \begin{matrix}
0 & 0 & d_{2}\frac{y_{+}-y_{-}}{w_{+}-w_{-}}& d_{1}\frac{z_{+}-z_{-}}{w_{+}-w_{-}}\\
        0 & 0 & 0& d_{2}\frac{z_{+}-z_{-}}{x_{+}-x_{-}}\\
        0 & 0 & 0& 0 \\
        0 & 0 & 0& 0\\
        \end{matrix}\right),$$

$$A_{2}=\left( \begin{matrix} -b_{11} & d_{1}\frac{x_{+}-x_{-}}{w_{+}-w_{-}} & d_{2}\frac{y_{+}-y_{-}}{w_{+}-w_{-}}& 0\\
        d_{1}\frac{w_{+}-w_{-}}{x_{+}-x_{-}} & -b_{22} & d_{1}\frac{y_{+}-y_{-}}{x_{+}-x_{-}}& d_{2}\frac{z_{+}-z_{-}}{x_{+}-x_{-}}\\
        d_{2}\frac{w_{+}-w_{-}}{y_{+}-y_{-}} & d_{1}\frac{x_{+}-x_{-}}{y_{+}-y_{-}} & -b_{33}& d_{1}\frac{z_{+}-z_{-}}{y_{+}-y_{-}} \\
         0 & d_{2}\frac{x_{+}-x_{-}}{z_{+}-z_{-}} & d_{1}\frac{y_{+}-y_{-}}{z_{+}-z_{-}}& -b_{44}\\
         \end{matrix}\right),$$
       $$A_{3}=\left( \begin{matrix}
        0 & 0 & 0& 0\\
        0 & 0 & 0& 0\\
        d_{2}\frac{w_{+}-w_{-}}{y_{+}-y_{-}} & 0 & 0& 0 \\
        d_{1}\frac{w_{+}-w_{-}}{z_{+}-z_{-}} & d_{2}\frac{x_{+}-x_{-}}{z_{+}-z_{-}} & 0& 0\\
        \end{matrix}\right),$$
with    $b_{ii}$ (i=1,2,3,4) is such that $A_{1}+A_{2}+A_{3}=0$. Note that $(A_{1}+A_{2}+A_{3}) \mathbf{\vec{1}}=0$, since $\mathbf{\vec{1}}$ is an equilibrium, we have $f_{i}(1)=0$. Obviously, $f_{i}(0)=0$. Thus, $f_{i}(\xi)$ is of form $f_{i}(\xi)=k_{i}\xi(\xi-c_{i})(\xi-1)$ for some $k_{i}, c_{i} \in \RR$.

For simplicity, let $x(\xi)=(w(\xi),x(\xi),y(\xi),z(\xi))^{T}$ and $F(x)=(f_{1}(w),f_{2}(x),f_{3}(y),f_{4}(z))^{T}$, then we consider the following system of equations in

\vspace{-.1in}\begin{equation}
\label{4D-eq}
cx'(\xi)-\sum_{j=1}^{3} A_{j}(\xi)x(\xi+r_{j})+F(x)=0, x(-\infty)=\vec{0}, x(\infty)=\vec{1},
\vspace{-.1in}\end{equation}
where $r_{1}=-h$, $r_{2}=0$ and $r_{1}=h$. \\

Choose the equilibria with $w_{+}-w_{-}>0, x_{+}-x_{-}>0, y_{+}-y_{-}<0$, and $z_{+}-z_{-}<0$  (or with $w_{+}-w_{-}<0, x_{+}-x_{-}<0, y_{+}-y_{-}>0$, and $z_{+}-z_{-}>0$). Without loss of generality, assume $d_{1}<0$. We rewrite the system of equations (\ref{4D-eq}) as the following:
\vspace{-.1in}\begin{equation}
\label{4D-system}
cx'(\xi)-\sum_{j=1}^{3} \tilde{A}_{j}(\xi)x(\xi+r_{j})+\sum_{j=1}^{3} B_{j}(\xi)x(\xi+r_{j})+F(x)=0, x(-\infty)=\vec{0}, x(\infty)=\vec{1}
,
\vspace{-.1in}\end{equation}
where $r_{1}=-h$, $r_{2}=0$ and $r_{1}=h$, \\
$$\tilde{A}_{1}=\left( \begin{matrix}
0 & 0 & d_{2}\frac{y_{+}-y_{-}}{w_{+}-w_{-}}& d_{1}\frac{z_{+}-z_{-}}{w_{+}-w_{-}}\\
        0 & 0 & 0& d_{2}\frac{z_{+}-z_{-}}{x_{+}-x_{-}}\\
        0 & 0 & 0& 0 \\
        0 & 0 & 0& 0\\
        \end{matrix}\right),$$

$$\tilde{A}_{2}=\left( \begin{matrix} -\tilde{b}_{11} & 0 & d_{2}\frac{y_{+}-y_{-}}{w_{+}-w_{-}}& 0\\
        0 & -\tilde{b}_{22} & 0& d_{2}\frac{z_{+}-z_{-}}{x_{+}-x_{-}}\\
        d_{2}\frac{w_{+}-w_{-}}{y_{+}-y_{-}} & 0 & -\tilde{b}_{33}& d_{1}\frac{z_{+}-z_{-}}{y_{+}-y_{-}} \\
         0 & d_{2}\frac{x_{+}-x_{-}}{z_{+}-z_{-}} & d_{1}\frac{y_{+}-y_{-}}{z_{+}-z_{-}}& -\tilde{b}_{44}\\
         \end{matrix}\right),$$
       $$\tilde{A}_{3}=\left( \begin{matrix}
        0 & 0 & 0& 0\\
        0 & 0 & 0& 0\\
        d_{2}\frac{w_{+}-w_{-}}{y_{+}-y_{-}} & 0 & 0& 0 \\
        d_{1}\frac{w_{+}-w_{-}}{z_{+}-z_{-}} & d_{2}\frac{x_{+}-x_{-}}{z_{+}-z_{-}} & 0& 0\\
        \end{matrix}\right),$$
        and $B_{1}=B_{3}=0$

        $$B_{2}=d_{1}\left( \begin{matrix} -\frac{x_{+}-x_{-}}{w_{+}-w_{-}}            & \frac{x_{+}-x_{-}}{w_{+}-w_{-}}            & 0        & 0\\
        \frac{w_{+}-w_{-}}{x_{+}-x_{-}}             & -\frac{y_{+}-y_{-}}{x_{+}-x_{-}}-\frac{w_{+}-w_{-}}{x_{+}-x_{-}}      & \frac{y_{+}-y_{-}}{x_{+}-x_{-}}        & 0\\
        0                          & \frac{x_{+}-x_{-}}{y_{+}-y_{-}}               & -\frac{x_{+}-x_{-}}{y_{+}-y_{-}}                     & 0 \\
         0 & 0                & 0                                                                                  & 0\\
         \end{matrix}\right),$$ with    $\tilde{b}_{ii}$ (i=1,2,3,4) is such that $\tilde{A}_{1}+\tilde{A}_{2}+\tilde{A}_{3}+B_{2}=0$.

         Consider the reference system,
\vspace{-.1in}\begin{equation}
\label{Ref-4D-eq}
cx'(\xi)-\sum_{j=1}^{3} \tilde{A}_{j}(\xi)x(\xi+r_{j})+F(x)=0, x(-\infty)=\vec{0}, x(\infty)=\vec{1},
\vspace{-.1in}\end{equation}
where $r_{1}=-h$, $r_{2}=0$ and $r_{1}=h$. \\

We can pick those equilibria $(w^{\pm}, x^{\pm},y^{\pm}, z^{\pm})$ such that after the transformation to $\vec{0}$ and $\vec{1}$, any other 4-periodic state
$\vec{\phi}=\{\phi_{n}\}_{n \in \ZZ}$ with $\phi_{n} \in (0, 1)$, if it exists, is unstable. In this paper, we focus on the cases having bistable dynamics after the transformation. By Theorem \ref{CGW-prop}, there exists a traveling wave solution $(c_{0}, \phi_{0})$ for (\ref{Ref-4D-eq}).

Denote $\Lambda x:=\sum_{j=1}^{3} A_{j}(\xi)x(\xi+r_{j}).$ Let $B x:=\sum_{j=1}^{3} B_{j}(\xi)x(\xi+r_{j})$ and $\Lambda_{\epsilon}= \Lambda+\epsilon B$. We consider the following:
\vspace{-.1in}\begin{equation}
\label{Perturbed-4D-eq} cx'-\Lambda_{\epsilon} x+F(x)=0, x(-\infty)=\vec{0}, x(\infty)=\vec{1}.
\vspace{-.1in}\end{equation}
%

Let  $L_{0}^{+}\phi:=c_{0}\phi'- \Lambda_{0} \phi+ \gamma(\phi_{0})\phi $ and $L_{0}^{-}\phi:=-c_{0}\phi'- \Lambda_{0}^{*}\phi+ \gamma(\phi_{0})\phi $, where $\Lambda_{0}^{*}\Psi$ is the adjoint operator of $\Lambda_{0}$.  We assume that
\vskip 1em

\noindent{\bf (D1)} {\it $L_{0}^{\pm}$ are asymptotically hyperbolic.}

\begin{theorem}{\it
 \label{4D-existence-thm}
Suppose $c_{0}\neq 0$. Assume $(D1)$. Then there exists a positive constant $\epsilon^{*}$ such that for every $\epsilon \in (0,\epsilon^{*}]$, the problem \eqref{Perturbed-4D-eq} admits a solution $(c_{\epsilon},\phi_{\epsilon})$ satisfying
$$\lim_{c_{\epsilon} \to c_{0}}(c_{\epsilon},\phi_{\epsilon})=(c_{0},\phi_{0}).$$
}
\end{theorem}

By Remark \ref{extensionremark}, we can extend $\epsilon^{*}$ if $(H4)$ is satisfied. If we can extend $\epsilon^{*}$ to 1, then we successfully get the existence of traveling wave solution for Equation \eqref{4D-system}. We remark that, unlike the previous 2D case, (H4) is not automatically satisfied for this 4D system. At the end of each extension, we have to check the assumption (H4). Let $L_{\epsilon}(\pm\infty)\phi:=c_{\epsilon}\phi'- \Delta_{\epsilon}\phi+ \gamma^{\pm} \phi $ and $L_{\epsilon}^{*}(\pm\infty)\phi:=-c_{\epsilon}\phi'- \Delta_{\epsilon}^{*}\phi+ \gamma^{\pm} \phi $, where $\gamma^{\pm}=\begin{pmatrix} \gamma^{\pm}_{1} & 0 & 0 & 0\\
        0 &\gamma^{\pm}_{2}& 0& 0\\
        0 & 0&\gamma^{\pm}_{3}& 0\\
        0& 0& 0&\gamma^{\pm}_{4}
        \end{pmatrix}$ for i=1,2,3,4 with $\gamma^{+}_{i}=f_{i}'(1),\gamma^{-}_{i}=f_{i}'(0)$ for i=1,2,3,4. (H4) is equivalent to the following:
\vskip 1em
\noindent{\bf ($\hat{D}1$)} {\it $L_{\epsilon}(\pm\infty)$ and $L_{\epsilon}^{*}(\pm\infty)$ are hyperbolic.}

\begin{remark} In this section we have considered $\tilde A_j$ such that the results
in \cite{CGW} on existence of traveling waves for bistable problems give
monotone waveforms for the limiting system.
This yields, via the results in \cite{HVV}, a one dimensional kernel for the linearization
about the reference solution. Alternatively, if perturbations include all terms multiplying
$d_1$, then the limiting system is decoupled and (A4) is not satisfied. However, by
considering the even and odd systems independently, the linearization about the reference
solution has a two dimensional kernel and the behavior of solutions under perturbation
may be analyzed using the bifurcation equations obtained through the Lyapunov-Schmidt
reduction.

\end{remark}

\subsection{Traveling Waves for LDEs with Infinite-Range Interactions}
In this section, we study the a generalized model  of \cite{CGW} by adding some infinite range interactions. Consider the following:
 \vspace{-.1in}\begin{equation}
\label{Infinity-CGW-eq}
\dot{u}_{n}(t)=\sum_{k}a_{n,k}u_{n+k}(t)+f_{n}(u_{n}(t)), n \in \ZZ, t>0, \vspace{-.1in}
\end{equation}
where the coefficients $a_{n,k}$ are real numbers satisfying $\sum_{k}a_{n,k}e^{k\lambda}<\infty$ for any $\lambda \in \RR$ and satisfy the assumptions (A1,A2,A4,A5). Compared with the equation in \cite{CGW}, the essential difference is in (A3) and (A5), where we remove the
assumption (A3), finite range interactions, and consider an infinite sum in (A5).

Consider the finite range interaction problem as in \cite{CGW},
\vspace{-.1in}\begin{equation}
\label{Infinity-CGW-refeq}
\dot{u}_{n}(t)=\sum_{0<|k|\leq k_{0}}a_{n,k}[u_{n+k}(t)-u_{n}(t)]+f_{n}(u_{n}(t)), n \in \ZZ, t>0. \vspace{-.1in}
\end{equation}
By Theorem \ref{CGW-prop}, there exists a traveling wave solution $(c_{0}, \phi_{0})$ for (\ref{Infinity-CGW-refeq}).
Let $$(\Lambda u)_{n}(t):=-c u_{n}'(t)+\sum_{0<|k|\leq k_{0}} a_{n,k}[u_{n+k}(t)-u_{n}(t)]$$ and $$(B u)_{n}(t):=\sum_{|k| > k_{0}}a_{n,k}[u_{n+k}(t)-u_{n}(t)].$$ Let $\Lambda_{\epsilon}= \Lambda+\epsilon B$. Then we have the perturbed equation of \eqref{Infinity-CGW-refeq},
 \vspace{-.1in}\begin{equation}
\label{CGW-Perturbed-eq}
(\Lambda_{\epsilon} u)_{n}(t)+f_{n}(u_{n}(t))=0,u_{n}(+\infty)=1 \  and \ u_{n}(-\infty)=0, n \in \ZZ, t>0. \vspace{-.1in}
\end{equation}

Let $(L_{0}^+ u)_{n}(t):=(\Lambda u)_{n}(t)+f'_{n}((\phi_{0})_{n})u_{n}(t)$. Let $$(\hat{L}_{\infty}^{+}\phi)_{n}(t):=(\Lambda u)_{n}(t)+ f'_{n}(1)\phi$$ and $$(\hat{L}_{-\infty}^{+}\phi)_{n}(t):=(\Lambda u)_{n}(t)+ f'_{n}(0)\phi.$$ Then we have their adjoint operators, denoted by $\hat{L}_{\infty}^{-}\phi$ and $\hat{L}_{-\infty}^{-}\phi$. We make an assumption,\\
\medskip
\noindent{\bf ($E1$)} {\it
$\hat{L}_{\infty}^{\pm}$ and $\hat{L}_{-\infty}^{\pm}$ are hyperbolic.}
\medskip

Let $\ds(B_{k_{0}}\phi)_{i}:=\sum_{|k|< k_{0}}a_{n,k}e^{k\mu}\phi_{i+k}$ for given $\mu \in \RR$. Consider the eigenvalue problem:
\vspace{-.1in}\begin{equation}
\label{BCC-eigen-eq}
\lambda \phi_{i}=(B_{k_{0}}\phi)_{i} + L_{i} \phi_{i}
\vspace{-.1in}\end{equation}
with $\phi_{i+n}=\phi_{i}\geq 0$, $\|\phi\|_{\infty}=1$ and $L_{i} \in \RR$.

\begin{lemma}
For each $k_{0}$, if $B_{k_{0}}$ is irreducible and quasipositive(i.e, off-diagonal elements are nonnegative), then principal eigenvalue exists, denoted by $\lambda(k_{0})$. Moreover, if both $\lambda(k_{0})$ and $\lambda(\infty)$ exist, $\ds\lim_{k_{0} \to \infty }\lambda(k_{0})= \lambda(\infty)$.
\end{lemma}
\begin{proof}
The existence of a principal eigenvalue is followed by Krein-Rutman theorem. Moreover, we have that $\ds\lambda(k_{0})=\lim_{n \to \infty}\|B_{k_{0}}^{n}\|^{1/n}$, which implies that $\ds\lim_{k_{0} \to \infty }\lambda(k_{0})= \lambda(\infty)$.
\end{proof}

Let $(\lambda_{0},\{\phi_{i}^{0}\})$ and $(\lambda_{1},\{\phi_{i}^{1}\})$ be the corresponding principal eigenvalue and eigenvectors for $L_{i}=f_{i}'(0)$ and $f_{i}'(1)$ respectively.

\begin{lemma}\label{lemma4.5}
A traveling wave must have exponential tails:
$$\lim_{i-ct \to -\infty}\frac{u_{i}(t)}{e^{(i-ct)\lambda_{0}}\phi_{i}^{0}}=h^{-},\lim_{i-cti-ct \to \infty}\frac{u_{i}(t)}{e^{(i-ct)\lambda_{1}}\phi_{i}^{1}}=h^{+},$$
where $(\lambda_{0},\{\phi_{i}^{0}\})$ and $(\lambda_{1},\{\phi_{i}^{1}\})$ are the corresponding principal eigenvalue and eigenvectors for eigenvalue problem \eqref{BCC-eigen-eq}.
\end{lemma}
\begin{proof}
This can be proved by modifying the arguments (replacing $k_{0}$, that defines the
finite range of interactions, with n, the period of the media)
in the proof of Theorem 2 of \cite{CGW}.
\end{proof}
Then we have the following theorem.
\begin{theorem} {\it
\label{IR-Thm} Assume that (E1) and $\vec{0}$ and $\vec{1}$ are steady-states and any other N-periodic state
$\vec{\phi}=\{\phi_{n}\}_{n \in \ZZ}$ with $\phi_{n} \in (0, 1)$, if it exists, is unstable. Then
\begin{itemize}
\item[(1)] There exists an $\epsilon^*$ such that for all $\epsilon \in (0,\epsilon^*]$, the problem (\ref{CGW-Perturbed-eq}) admits a solution
$(c, \vec{w})$ satisfying
$\vec{w}(-\infty) =\vec{0} < \vec{w}(\xi) < \vec{1} = \vec{w}(+\infty)$ for all $\xi \in \RR.$
\item[(2)] If for some positive integer $k_{0}$ and $\Pi(k_0)$ which is  such that
$\ds\sum_{|k|>k_{0}}a_{n,k} <\Pi(k_0)$ and $K_{2}<1/C_{0}$ in Lemma \ref{estimate-lm33}, the problem (\ref{Infinity-CGW-eq}) admits a solution $(c, \vec{w})$ satisfying
$\vec{w}(-\infty) =\vec{0} < \vec{w}(\xi) < \vec{1} = \vec{w}(+\infty)$ for all $\xi \in \RR.$
\end{itemize}}
Moreover, for $0<\epsilon\ll 1$, $c\dot{\phi}_{n}(\xi)<0$ for $c \neq 0$, $n \in \ZZ$ and $\xi \in \RR.$
\end{theorem}
\begin{proof}
The existence of traveling waves follows from the arguments in Section 3.
Next we show that monotonicity persists under small perturbations. By the arguments in Theorem 2 of \cite{CGW} (see Lemma \ref{lemma4.5}), a traveling wave must have exponential tails:
$$\lim_{i-ct \to -\infty}\frac{u_{i}(t)}{e^{(i-ct)\lambda_{0}}\phi_{i}^{0}}=h^{-},\lim_{i-cti-ct \to \infty}\frac{u_{i}(t)}{e^{(i-ct)\lambda_{1}}\phi_{i}^{1}}=h^{+},$$
where $(\lambda_{0},\{\phi_{i}^{0}\})$ and $(\lambda_{1},\{\phi_{i}^{1}\})$ are the corresponding principal eigenvalue and eigenvectors for $L_{i}=f_{i}'(0)$ and $f_{i}'(1)$ respectively:
$$\mu \phi_{i}=\Sigma_{k}a_{n,k}\phi_{i+k} + L_{i} \phi_{i},$$ with $\phi_{i+n}=\phi_{i}\geq 0$, $\|\phi\|_{\infty}=1$. Dividing the Equation \eqref{CGW-Perturbed-eq} by $e^{(i-ct)\lambda_{0}}\phi_{i}^{0}$ and taking the limit for $i-ct \to - \infty$, we have $\ds\lim_{i-ct \to -\infty}\frac{\frac{\partial u_{i}(t)}{\partial t}}{e^{(i-ct)\lambda_{0}}\phi_{i}^{0}}=\lambda_{0}h^{-}$. Similarly we have $\ds-\lim_{i-ct \to \infty}\frac{\frac{\partial u_{i}(t)}{\partial t}}{e^{(i-ct)\lambda_{1}}\phi_{i}^{1}}=\lambda_{1}h^{+}$. Note that $\lambda_{0}>0$ and $\lambda_{1}<0$. We have that $\frac{\partial u_{i}(t)}{\partial t}$ has the same sign as $|i-ct|>M$ for some large M. Thus the traveling wave will preserve the  monotonicity at the two far ends for small perturbation because the principal eigenvalue will preserve the sign for small perturbation. Obviously, $\frac{\partial u_{i}(t)}{\partial t}$ will preserve the sign on $i-ct \in [-M,M]$ for small perturbation. This completes the proof.
\end{proof}

\section{Conclusion}\label{conclusion}
In this paper we develop an existence theory via perturbation arguments for
traveling wave solutions of vector lattice differential equations. Motivation
comes from problems in which there is not a comparison principle. In particular,
we consider lattice differential equations in which there are repelling first
and/or second nearest neighbor interactions. The structure of the kernel
(see Proposition 8.2 in \cite{HVV})
of the linearized operator of the limiting system is central to our
analysis. Our general result is modeled after the perturbation arguments in
\cite{BCC}. A possible alternative approach is the Newton/Lyapunov-Schmidt method
developed in \cite{HL1, HVV}. Finally, we employ the technique developed here to
show the existence of traveling waves for bistable lattice differential equations
in periodic media with infinite range interactions. Although the results obtained
here are primarily of a local nature, they may be extended to global continuation
results in certain cases. This necessitates a Fredholm theory for linearized operators
that do not satisfy a strict ellipticity conditions such as (A5), e.g., see \cite{BCC},
together with results on the dimension and structure of the kernel.
While the Fredholm theory for problems with infinite
range interactions is not well developed, the results
in \cite{LVV} apply to certain infinite range interactions.

\end{document}